\documentclass{ovidius2}
\usepackage{amsmath,amssymb,amsfonts,amsthm,epsfig,color,graphicx,eucal,mathrsfs,enumitem,footmisc}
\newtheorem{theorem}{Theorem}
\newtheorem{lemma}[theorem]{Lemma}
\usepackage{algorithm}
\usepackage{algpseudocode}

\title{INEXACT LINEAR SOLVES IN MODEL REDUCTION OF BILINEAR DYNAMICAL SYSTEMS}
\author{}
\MathReviews{Primary 34D10, 34D20, 41A05, 65F10.}

\firstpagenumber{00}
\begin{document}
\edef\slashcatcode{\the\catcode`\/}
\catcode`\/12
\newdimen\bigpoint \bigpoint 1bp
\newdimen\paperwidth
\newdimen\paperheight
\newcount\paperwidthcount
\newcount\paperheightcount
%
\def\clearbophook{%
  \special{!userdict begin /bop-hook {} def end }}
%
\def\PSprocappend#1#2{
  [#1 where {pop #1 load aload pop} if {#2} aload pop] cvx #1 exch def }
%
\def\papersize#1#2{%
  \paperwidth #1 \divide\paperwidth by \bigpoint
  \paperwidthcount\paperwidth
  \paperheight #2 \divide\paperheight by \bigpoint
  \paperheightcount\paperheight }
%
\def\mirror{%
  \special{!userdict begin \PSprocappend {/bop-hook}
    {[-1 0 0 1 \the\paperwidthcount\space 0] concat} end }}
%
\def\landscape{%
  \special{!userdict begin \PSprocappend {/bop-hook}
     {[0 -1 1 0 \the\paperwidthcount\space \the\paperheightcount\space sub
     \the\paperheightcount] concat} end }}
%
\def\hpfour{%
}
%
\def\userbophook#1{%
  \special{!userdict begin \PSprocappend {/bop-hook}{#1} end }}
%
\papersize{210mm}{297mm}
\catcode`\/\slashcatcode
%

\maketitle

\thispagestyle{plain}
\begin{center}
{\bf Rajendra Choudhary and Kapil Ahuja}
\end{center}
\vspace{3mm}

\begin{abstract}
	Bilinear dynamical systems are commonly used in science and engineering because they form a bridge between linear and non-linear systems. However, simulating them is still a challenge because of their large size. Hence, a lot of research is currently being done for reducing such bilinear dynamical systems (termed as bilinear model order reduction or bilinear MOR). Bilinear iterative rational Krylov algorithm (BIRKA) is a very popular, standard and mathematically sound algorithm for bilinear MOR, which is based upon interpolatory projection technique. An efficient variant of BIRKA, Truncated BIRKA (or TBIRKA) has also been recently proposed. 
	\par 
	Like for any MOR algorithm, these two algorithms also require solving multiple linear systems as part of the model reduction process. For reducing very large dynamical systems, which is now-a-days becoming a norm, scaling of such linear systems with respect to input dynamical system size is a bottleneck. For efficiency, these linear systems are often solved by an iterative solver, which introduces approximation errors. Hence, stability analysis of MOR algorithms with respect to inexact linear solves is important. In our past work, we have shown that under mild conditions, BIRKA is stable (in the sense as discussed above). Here, we look at stability of TBIRKA in the same context.
	\par 
	Besides deriving the conditions for a stable TBIRKA, our other novel contribution is the more intuitive methodology for achieving this. This approach exploits the fact that in TBIRKA a bilinear dynamical system can be represented by a finite set of functions, which was not possible in BIRKA (because infinite such functions were needed there). The stability analysis techniques that we propose here can be extended to many other methods for doing MOR of bilinear dynamical systems, e.g., using balanced truncation or the ADI methods.	
\end{abstract}

\vskip.5cm
\section{Introduction}\label{sec:Inroduction}
A dynamical system, usually represented by a set of differential equations, can be linear or non-linear \cite{Mathematics_and_Climate_dynamicalsystembook,Differential_Dynamical_Systems_book,dynamical_system_application_1986}. Linear dynamical systems have been studied more than the non-linear ones because of the obvious ease in working with them. Bilinear dynamical systems form a good bridge between the linear and the non-linear cases, and are usually approximated by a varying degree of bilinearity \cite{nonlinear_to_bilinear_paper_1974,RughnonlinearsystemBook,nonlinear_to_bilinear_paper_2015}. In this manuscript, we focus on bilinear dynamical systems.
\par 
Dynamical systems coming from different real world applications are very large in size. Thus, simulation and computation with such systems is prohibitively expensive. Model Order Reduction (MOR) techniques provide a smaller system that besides being cheaper to work with, also replicates the input-output behaviour of the original system to a great extent \cite{AntoulasBook,Peterbookdynamicalsystemreduction,grimmephdthesis,Zbaimorpaper,k_willcox_model_reduction_paper,serkanIRKAsiampaper,mimosystem,biliner_model_reduction_Zbai,BIRKApaper,AIRGApaper}. Since, bilinear dynamical systems have been recently studied, 
the techniques for reducing them are also recent. 
\par 
Out of the many methods available for performing bilinear MOR \cite{biliner_model_reduction_Zbai,BIRKApaper,bilinear_model_reduction_Mian_Ilyas,bilinear_model_reduction_peter_benner_balanced_trunc,bilinear_model_reduction_Antoulas_Loewner_framework,bilinear_model_reduction_kang-li_xu_gramian_based,garretphdthesis,TBIRKApaper}, we focus on a commonly used interpolatory projection method. BIRKA (Bilinear Iterative Rational Krylov Algorithm) \cite{BIRKApaper} is a very popular algorithm based upon this technique for reducing \textit{first-order} bilinear dynamical systems\footnote{\textit{First-order} implies that the highest derivative of the state variable in the dynamical system is one. \textit{Second-order} and \textit{higher-orders} are similarly defined.}. BIRKA's biggest drawback is that it does not scale well in time (with respect to increase in the size of the input dynamical system). A cheaper variant of BIRKA, called TBIRKA (Truncated Bilinear Iterative Rational Krylov Algorithm) \cite{garretphdthesis,TBIRKApaper} has also been proposed. 
\par 
Like in any other MOR algorithm, in BIRKA and TBIRKA also, people often use direct methods like LU-factorization, etc., to solve the arising linear systems, which are too expensive (i.e., computational complexity of $\mathcal{O}(n^3)$, where $n$ is the original system size). A common solution to this scaling problem is to use iterative methods like the Krylov subspace methods, etc., which have a reduced computational complexity (i.e., $\mathcal{O}(n \times nnz)$, where $nnz$ is the number of nonzeros in the system matrix) \cite{greenbaum_book,SaaditerativeBook}. Although iterative methods are cheap, they are inexact too. Hence, studying stability of the underlying MOR algorithm (here BIRKA and TBIRKA) with respect to such approximate linear solves becomes important \cite{TrefethenBauBook,Demmelbook}. 
\par 
One of the first works that performed such a stability analysis focused on popular MOR algorithms for \textit{first-order} linear dynamical systems \cite{sarahinexactlaapaper}. Here, the authors briefly mention that their analysis would be easily carried from the \textit{first-order} to the \textit{second-order} case. A detailed stability analysis focusing on reducing \textit{second-order} linear dynamical systems has been done
in \cite{NavSISC}. A different kind of stability analysis for MOR of \textit{second-order} linear dynamical systems has been done in \cite{TOAR_stability_analysis_ZBai}. In this, the authors first show that the SOAR algorithm (second order Arnoldi) is unstable with respect to the machine precision errors (and not inexact linear solves). Then, they propose a Two-level orthogonal Arnoldi (TOAR) algorithm that cures this instability of SOAR. An extended stability analysis for BIRKA (as above, a popular MOR algorithm for \textit{first-order} bilinear dynamical systems) has been recently done in \cite{BIRKAstabilitylaapaper}. For rest of this manuscript, whenever stability analysis is referred, we mean it with respect to inexact linear solves. 
\par 
We follow the stability analysis framework of BIRKA from \cite{BIRKAstabilitylaapaper} and propose equivalent theorems for TBIRKA. The approach here is slightly different, which forms our most \textit{novel} contribution. Norm of the dynamical system plays an important role in stability analysis (the kind of norm is discussed later). In BIRKA stability analysis, a single expression for bilinear dynamical system norm is used (involving a Volterra series). In TBIRKA stability analysis, a similar single expression (involving a truncated Volterra series) leads to complications. Alternatively, in TBIRKA, because of truncation, the bilinear dynamical system can be represented by a finite set of functions. This was not possible in BIRKA where infinite such functions were needed. Thus, in TBIRKA stability analysis, we use norm of all such functions leading to easier derivations. Our stability analysis, as done for BIRKA earlier and for TBIRKA here, can be easily extended to other MOR algorithms for bilinear dynamical systems, e.g., projection based \cite{biliner_model_reduction_Zbai}, implicit Volterra series \cite{bilinear_model_reduction_Mian_Ilyas}, balanced truncation \cite{bilinear_model_reduction_peter_benner_balanced_trunc},  gramian based \cite{bilinear_model_reduction_kang-li_xu_gramian_based}, etc.
\par 
The rest of the paper is divided into three more parts. In Section \ref{sec:Background}, we first give a brief overview of MOR for bilinear dynamical systems using a projection method. Next, we review the stability analysis of BIRKA from \cite{BIRKAstabilitylaapaper}. A stability analysis typically involves satisfying two conditions. Hence, finally in this section, we study the first condition for a stable TBIRKA. We analyze the second condition of stability in TBIRKA's context in Section \ref{sec:Stability_Analysis}. In Section  \ref{sec:Conclusions Future Work}, we give conclusions and discuss the future work. For the rest of this paper, we use the terms and notations as listed below.
\begin{enumerate}
	\item [a.] The $H_2-$norm is a functional norm defined as \cite{sarahinexactlaapaper,garretphdthesis,TBIRKApaper}
	\begin{alignat}{2}\label{eq:H_2normdefinitionsubsystem}
	\left \| H_k \right \|_{H_2}^2 = \left ( \frac{1}{2\pi } \right )^k \int_{-\infty}^{\infty} \ldots  \int_{-\infty}^{\infty} \left \| H_k\left ( i\omega_1, \ \ldots, \ i \omega_k  \right ) \right \|_F^2 d\omega_1\ldots d\omega_k, 
	\end{alignat}
	where $i$ denotes $\sqrt{-1}$. Here, we assume that all $H_2-$norms computed further exist. In other words, the improper integrals defined by the $H_2-$norm give finite value. This is a reasonable assumption because this happens often in practice (see \cite{sarahinexactlaapaper}, where stability analysis of IRKA is done as well as \cite{BIRKAstabilitylaapaper}, where stability analysis of BIRKA is done). 
	\item[b.] The $H_\infty-$norm is also a functional norm, defined as \cite{sarahinexactlaapaper,garretphdthesis,TBIRKApaper}
	\begin{alignat*}{2}
	\left \| H_k \right \|_{H_\infty} = \underset{\omega_1,\ \ldots, \ \omega_k \in \mathbb{R}}{max} \left \| H_k \left (  i \omega_1,\ \ldots, \ i\omega_k \right ) \right \|_2. \notag
	\end{alignat*}
	\item [c.] The Kronecker product between two matrices $P$ (of size $m\times n$), and $Q$ (of size $s \times t$) is defined as
	\begin{equation}
	P\otimes Q =\begin{bmatrix}
	p_{11}Q & \cdots  & p_{1n}Q\\ 
	\vdots  & \ddots  & \vdots\\ 
	p_{m1}Q & \cdots & p_{mn}Q 
	\end{bmatrix}, \notag
	\end{equation}
	where $p_{ij}$ is an element of matrix $P$ and order of $P\otimes Q$ is $ms\times nt$. 
	\item [d.] $vec$ operator on a matrix $P$ is defined as
	\begin{equation}
	vec(P)= \left ( p_{11},\ \ldots,\ p_{m1},\ p_{12},\ \ldots,\ p_{m2},\ \ldots \ \ldots,\ p_{1n},\ \ldots,\ p_{mn} \right )^T. \notag
	\end{equation}
	\item [e.] Also, $I_n$ denotes an identity matrix of size $n \times n$ and $\mathbb{R}$ denotes the set of real numbers.
\end{enumerate}
\section{Background}\label{sec:Background}
A \textit{first-order} bilinear dynamical system is usually represented as \cite{biliner_model_reduction_Zbai,BIRKApaper}
\begin{equation}
\zeta: \left\{
\begin{array}{ll}\label{eq:originalbilinearequation}
\dot{x}(t) = Ax(t) +  N x(t) u(t)+bu(t), \\
y(t) = cx(t), 
\end{array}
\right.
\end{equation}
where $ A,N  \in \mathbb{R}^{n \times n}$, $b\in\mathbb{R}^{n \times 1} \ \textnormal{and} \ c \in  \mathbb{R}^{1\times n} $. Also, $u(t):\mathbb{R} \rightarrow \mathbb{R}, \ y(t):\mathbb{R} \rightarrow \mathbb{R} \  \textnormal{and} \ x(t): \mathbb{R} \rightarrow \mathbb{R}^n$, represent input, output and state of the bilinear dynamical system, respectively. This is a Single Input Single Output (SISO) system, which we have chosen for ease of our analysis. We plan to look at Multiple Input Multiple Output (MIMO) systems as part of our future work.  
\par 
A bilinear dynamical system can also be represented by a series of transfer functions, i.e.,
\begin{align} \label{eq:bilinear_series_representation_infinity}
\zeta = \underset{k\rightarrow \infty }{Lim} \ \zeta^k, 
\end{align} 
where $\zeta^k = \left \{H_{1}\left ( s_{1} \right ),\ H_{2}\left ( s_{1},\ s_{2} \right ),\ \ldots, \ H_k\left(s_1, \ s_2,\ \ldots, \ s_k\right)  \right\}$. Here,\\ $H_k\left(s_1, \ s_2,\ \ldots, \ s_k\right) $ is called the $k^{th}$ order transfer function of the bilinear dynamical system and is defined as
\begin{alignat}{2}\label{eq:subsystem_transfer_function}
H_{k}\left ( s_{1},\ \ldots,\ s_{k} \right ) =c\left ( s_{k}I-A \right )^{-1}N \left ( s_{k-1}I-A \right )^{-1}\ \ldots  N\left ( s_{1}I-A \right )^{-1}b.
\end{alignat}
After reduction, the bilinear dynamical system \eqref{eq:originalbilinearequation} can be represented as
\begin{equation}
\zeta_r : \left\{
\begin{array}{ll} \label{eq:reducedbilinearequation}
\dot{x}_r(t)= A_r x_r (t) +  N_r x_r(t) u(t)+ b_r u(t), \\
y_r(t) = c_r x_r(t), 
\end{array}
\right.
\end{equation}
where $A_r,\ N_r \in \mathbb{R}^{r \times r}$, $ b_r \in \mathbb{R}^{r\times 1} \ \textnormal{and} \ c_r\in \mathbb{R}^{1\times r} \ $  with $r \ll n$. The main goal of model reduction is to approximate $\zeta$ by $\zeta_r$ in an appropriate norm, such that for all admissible inputs, $y_r(t)$ is nearly same to $y(t)$. 
\par
As mentioned earlier, we use interpolatory projection for performing model reduction. The two common and standard algorithms here, BIRKA and TBIRKA, use a Petrov-Galerkin projection. Let $\mathcal{V}_r$ and $\mathcal{W}_r$ be the two $r$-dimensional subspaces, such that $\mathcal{V}_r = Range (V_r)$ and $\mathcal{W}_r = Range (W_r)$, where $V_r \in \mathbb{R}^{n \times r}$ and $W_r \in \mathbb{R}^{n \times r}$ are matrices. Also, let $\left ( W_r^T V_r   \right )$ be invertible\footnote{Obtaining such an invertible matrix is not
	difficult \cite{BIRKApaper,TBIRKApaper}.}. Applying the projection $x(t)= V_r x_r(t)$, and enforcing the Petrov-Galerkin conditions \cite{BIRKApaper,TBIRKApaper} on the original bilinear dynamical system \eqref{eq:originalbilinearequation}, we get the reduced system as
\begin{alignat*}{2}
\begin{split}
W_r ^T & \left( V_r \dot{x}_r(t)-A V_r x_r(t) -  N V_r x_r(t) u(t)-  bu(t) \right) = 0,\\
y(t) & = c V_rx_r(t).
\end{split}	
\end{alignat*} 
Comparing the above two equations with their respective equations in \eqref{eq:reducedbilinearequation}, we get a relation between the original system matrices and the reduced system matrices, i.e., 
\begin{alignat*}{2}
\begin{split} 
A_{r}&= \left ( W_r^T V_r   \right )^{-1} W_{r}^{T} A V_{r},\ N_r= \left ( W_r^T V_r   \right )^{-1} W_{r}^{T} N V_{r}, \\  b_{r}&= \left ( W_r^T V_r   \right )^{-1} W_{r}^{T} b, \ \textnormal{and} \  c_{r}= c V_{r}.
\end{split}
\end{alignat*} 
\par 
One way of obtaining subspaces $\mathcal{V}_r$ and $\mathcal{W}_r$ is to use Volterra series interpolation. Further, to decide where to interpolate so as to obtain an optimal reduced model, an $H_2-$optimization problem is commonly solved (Theorem 4.7 from \cite{garretphdthesis}).
\begin{theorem}\cite{garretphdthesis}
	\label{Theorem:Interpolation_condition_infinity}
	Let $\zeta$ be a bilinear system of order n. Let $\zeta_r$ be an $H_2 -$ optimal approximation of order r. Then, $\zeta_r$ satisfies the following multi-point Volterra series interpolation conditions:
	\begin{alignat}{2}
	\sum_{k=1}^{\infty}\sum_{l_1=1}^{r} \ldots \sum_{l_{k}=1}^{r} \phi_{l_1,\ l_2,\ \ldots, \ l_{k}} H_k \left( -\lambda_{l_1},\ -\lambda_{l_2},\ \ldots,\ -\lambda_{l_k} \right) \  = \notag\\ \sum_{k=1}^{\infty}\sum_{l_1=1}^{r} \ldots \sum_{l_{k}=1}^{r} \phi_{l_1,\ l_2,\ \ldots,\ l_{k}}  H_{r_k} \left( -\lambda_{l_1},\ -\lambda_{l_2},\ \ldots,\ -\lambda_{l_k} \right),\quad \textnormal{and} \notag \\
	\sum_{k=1}^{\infty}\sum_{l_1=1}^{r} \ldots \sum_{l_{k}=1}^{r} \phi_{l_1,\ l_2,\ \ldots,\ l_{k}}                \left ( \sum_{j=1}^{k} \frac{\partial }{\partial s_j} H_k\left ( -\lambda_{l_1},\ -\lambda_{l_2},\ \ldots,\ -\lambda_{l_k}\right )  \right ) = \notag \\ \sum_{k=1}^{\infty}\sum_{l_1=1}^{r} \ldots \sum_{l_{k}=1}^{r} \phi_{l_1,\ l_2,\ \ldots,\ l_{k}}    \left (  \sum_{j=1}^{k} \frac{\partial }{\partial s_j}  H_{r_k}\left ( -\lambda_{l_1},\ -\lambda_{l_2},\ \ldots, \ -\lambda_{l_k}\right )    \right ),\notag	
	\end{alignat}
	where $\phi_{l_1,\ l_2,\ \ldots,\ l_{k}}$ and $\lambda_{l_1},\ \lambda_{l_2},\ \ldots,\ \lambda_{l_k}$ 	 are residues and poles of the transfer function $H_{r_k}$ associated with $\zeta_r$, respectively.	
\end{theorem}  
BIRKA is designed in such a way that at convergence, the conditions of Theorem \ref{Theorem:Interpolation_condition_infinity} are satisfied leading to a locally $H_2-$optimal reduced model. Algorithm \ref{BIRKAAlgo} lists BIRKA.
\begin{algorithm}
	\caption{BIRKA~\cite{BIRKApaper,garretphdthesis,TBIRKApaper}}
	\label{BIRKAAlgo}
	\begin{algorithmic}[1]
		\\ Given an input bilinear dynamical system $A,\ N , \ b,\ c$.
		\\ Select an initial guess for the reduced system as $\check{A},\ \check{N},\ \check{b},\ \check{c}$. Also select stopping tolerance $btol$.
		\\While $\left ( \textnormal{relative change in eigenvalues of} \ \check{A}  \geq btol \right )$ \begin{enumerate}
			\item[a.] $R \Lambda R^{-1} = \check{A},\ \check{\check{b}}= \check{b}^T R^{-T}, \ \check{\check{c}}=\check{c}R, \ \check{\check{N}}= R^T \check{N} R^{-T}$. 
			\item[b.] $vec\left ( \mathbf{V }\right ) = \left ( -\Lambda \otimes I_n - I_{r} \otimes A - \check{\check{N}}^T \otimes  N  \right )^{-1}   \left ( \check{\check{b}}^T \otimes b\right )$.
			\item[c.] $vec\left ( \mathbf{W}\right ) = \left ( -\Lambda \otimes I_n - I_{r} \otimes A^T - \check{\check{N}} \otimes N^T \right )^{-1} \left ( \check{\check{c}}^T \otimes c^T\right )$.
			\item[d.] $\mathbf{V}_r = orth\left ( \mathbf{V }  \right ) , \ \mathbf{W}_r = orth\left ( \mathbf{W}  \right ) $.
			\item[e.] $\check{A}= (\mathbf{W}^T_r \mathbf{V}_r)^{-1} \mathbf{W}^T_r A \mathbf{V}_r, \quad  \check{N}= \left ( \mathbf{W}^T_r \mathbf{V}_r \right )^{-1} \mathbf{W}^T_r N \mathbf{V}_r,$ \quad 
			\item[] $\check{b}=\left ( \mathbf{W}^T_r \mathbf{V}_r \right )^{-1} \mathbf{W}^T_r b, \quad \check{c}= c\mathbf{V}_r.$	
		\end{enumerate}
		\\ $A_r = \check{A}, \quad N_r = \check{N}, \quad b_r =\check{b}, \quad c_r = \check{c}$.
	\end{algorithmic}
\end{algorithm}
\par 
TBIRKA is similar to BIRKA in most aspects except that it performs a truncated Volterra series interpolation. Here, instead of $\zeta$ in \eqref{eq:originalbilinearequation}-\eqref{eq:bilinear_series_representation_infinity}, they work with $\zeta^M$, which is defined as 
\begin{align}\label{eq:bilinear_series_representation_truncated}
\zeta^M = \left\{ H_{1}\left ( s_{1} \right), \ H_{2}\left ( s_{1},\ s_{2} \right),\ H_{3}\left ( s_{1},\ s_{2},\ s_{3} \right),\ \ldots, \ H_{M}\left ( s_{1},\ \ldots, \ s_{M} \right ) \right \},
\end{align}
with $H_{k}\left ( s_{1},\ \ldots, \ s_{k} \right )$ for $k=1, \ \ldots, \ M$ is given by \eqref{eq:subsystem_transfer_function}. Equivalent of Theorem \ref{Theorem:Interpolation_condition_infinity} above is as follows (Theorem 4.8 from \cite{garretphdthesis}):
\begin{theorem}\cite{garretphdthesis}
	\label{Theorem:Interpolation_condition_M_term}
	Let $\zeta=(A, N, b, c)$ be an order $n$ bilinear system and $\zeta^M$ be the polynomial system determined by $\zeta$. Let $\zeta_r=(A_r, N_r, b_r, c_r)$ be a bilinear system of order $r$, and define $\zeta_r^M$ as the polynomial system determined by $\zeta_r$. Suppose that $\zeta_r^M$ is an $H_2-$optimal approximation to $\zeta^M$. Then $\zeta_r^M$ satisfies
	\begin{alignat}{2}
	\sum_{k=1}^{M}\sum_{l_1=1}^{r} \ldots \sum_{l_{k}=1}^{r} \phi_{l_1,\ l_2,\ \ldots, \ l_{k}} H_k \left( -\lambda_{l_1},\ -\lambda_{l_2},\ \ldots,\ -\lambda_{l_k} \right) \  = \notag\\ \sum_{k=1}^{M}\sum_{l_1=1}^{r} \ldots \sum_{l_{k}=1}^{r} \phi_{l_1,\ l_2,\ \ldots,\ l_{k}}  H_{r_k} \left( -\lambda_{l_1},\ -\lambda_{l_2},\ \ldots,\ -\lambda_{l_k} \right),\quad \textnormal{and} \notag 
	\end{alignat}
	\begin{alignat}{2}
	\sum_{k=1}^{M}\sum_{l_1=1}^{r} \ldots \sum_{l_{k}=1}^{r} \phi_{l_1,\ l_2,\ \ldots,\ l_{k}}                \left ( \sum_{j=1}^{k} \frac{\partial }{\partial s_j} H_k\left ( -\lambda_{l_1},\ -\lambda_{l_2},\ \ldots,\ -\lambda_{l_k}\right )  \right ) = \notag \\ \sum_{k=1}^{M}\sum_{l_1=1}^{r} \ldots \sum_{l_{k}=1}^{r} \phi_{l_1,\ l_2,\ \ldots,\ l_{k}}    \left (  \sum_{j=1}^{k} \frac{\partial }{\partial s_j}  H_{r_k}\left ( -\lambda_{l_1},\ -\lambda_{l_2},\ \ldots, \ -\lambda_{l_k}\right )    \right ),\notag	
	\end{alignat}
	where $\phi_{l_1,\ l_2,\ \ldots,\ l_{k}}$ and $\lambda_{l_1},\ \lambda_{l_2},\ \ldots,\ \lambda_{l_k}$ 	 are residues and poles of the transfer function $H_{r_k}$ associated with $\zeta_r^M$, respectively.	
\end{theorem}  
Algorithm \ref{TBIRKAAlgo} lists TBIRKA.
\par 
Both BIRKA and TBIRKA in turn require solving large sparse linear systems of equations. If we compare Algorithm \ref{BIRKAAlgo} and \ref{TBIRKAAlgo}, we realize that the cost of linear solvers at each step of the \texttt{While} loop in the former is $\mathcal{O}\left(nr \times nr\right)$ and in the latter is $\mathcal{O}\left(M\times nr \times nr \right)$. This makes it seem that TBIRKA is more expensive than BIRKA. However, TBIRKA is implemented in such a way that the Kronecker products are avoided leading to cost of linear solves at each step of the \texttt{While} loop to be $\mathcal{O}\left(M \times r \times \left(n \times n\right)\right)$. Thus, if $M < r$, which is usually the case (e.g, $M=3 \ \textnormal{or}\  4$ and $r \gg M$), then TBIRKA is more efficient than BIRKA. For further details on this see chapter 4 in \cite{garretphdthesis} and Section 5.3 in \cite{TBIRKApaper}. These implementation details do not affect our stability analysis, and hence, we use Algorithm \ref{TBIRKAAlgo} in the current form as our base.
\begin{algorithm}[h]
	\caption{TBIRKA~\cite{garretphdthesis,TBIRKApaper}}
	\label{TBIRKAAlgo}
	\begin{algorithmic}[1]
		\\ Given an input bilinear dynamical system $A,\ N, \ b,\ c$.
		\\ Select an initial guess for the reduced system as  $ \check{A},\ \check{N},\ \check{b},\ \check{c}$. Also select the truncation index $M$ and stopping tolerance $tbtol$.
		\\ While $\left ( \textnormal{relative change in eigenvalues of} \ \check{A}  \geq tbtol \right )$
		\begin{enumerate}
			\item[a.] $R \Lambda R^{-1} = \check{A},\ \check{\check{b}}= \check{b}^T R^{-T},  \ \check{\check{c}}=\check{c}R,\  \check{\check{N}}= R^T \check{N} R^{-T}$.
			\item[b.] Compute
			\begin{align}
			& vec\left ( \mathbf{V}_{1} \right ) = \left ( -\Lambda\otimes I_{n}  -I_{r}\otimes A  \right )^{-1} \left(\check{\check{b}}^T \otimes b\right),\notag \\
			& vec\left ( \mathbf{W}_{1} \right ) = \left ( -\Lambda\otimes I_{n}  -I_{r}\otimes A^{T}  \right )^{-1} \left(\check{\check{c}}^T \otimes c^T\right). \notag 
			\end{align}
			\item[c.] For $j = 2, \ldots, M$, solve
			\begin{align}
			& vec\left ( \mathbf{V}_{j} \right ) = \left ( -\Lambda\otimes I_{n}  -I_{r}\otimes A  \right )^{-1} \left ( \check{\check{N}}^T \otimes N \right )  vec\left ( \mathbf{V}_{j-1} \right ), \notag \\
			& vec\left ( \mathbf{W}_{j} \right ) = \left ( -\Lambda\otimes I_{n}  -I_{r}\otimes A^{T}  \right )^{-1}  \left ( \check{\check{N}} \otimes N^{T} \right )   vec\left ( \mathbf{W}_{j-1} \right ). \notag
			\end{align}
			\item[d.] $\mathbf{V} = \sum\limits_{j=1}^{M} \mathbf{V}_j , \ \mathbf{W} = \sum\limits_{j=1}^{M} \mathbf{W}_j $.
			\item[e.] $\mathbf{V}_r = orth\left ( \mathbf{V} \right ) , \ \mathbf{W}_r = orth\left ( \mathbf{W}   \right ) $.
			\item [f.] $\check{A}= (\mathbf{W}^T_r \mathbf{V}_r)^{-1} \mathbf{W}^T_r A \mathbf{V}_r,\quad \check{N}= \left ( \mathbf{W}^T_r \mathbf{V}_r \right )^{-1} \mathbf{W}^T_r N \mathbf{V}_r,$  
			\item [] $\check{b}=\left ( \mathbf{W}^T_r \mathbf{V}_r \right )^{-1} \mathbf{W}^T_r b, \quad \check{c}= c\mathbf{V}_r$.			
		\end{enumerate}		
		\\ $A_r = \check{A}, \quad N_{r} = \check{N}, \quad b_r =\check{b}, \quad c_r = \check{c}$.
	\end{algorithmic}
\end{algorithm}
\par    
As mentioned earlier, using iterative methods for solving such linear systems introduces approximation errors. We have done a detailed stability analysis of BIRKA with respect to the inexact linear solves in \cite{BIRKAstabilitylaapaper}, and we briefly revisit this next. Generally, accuracy is the metric that tells about the correctness in the output of an inexact algorithm. Due to unavailability of the exact output, it is not possible to determine accuracy \cite{TrefethenBauBook,BIRKAstabilitylaapaper}. A more easier metric is stability. Backward stability is one such notation, which says ``A backward stable algorithm gives exactly the right output to nearly the right input" \cite{TrefethenBauBook}. 
In our context, theoretically we obtain two reduced systems. One by applying an inexact MOR algorithm (with iterative linear solves) on the original full model, and other by applying the same MOR algorithm but exactly (with direct linear solves) on a perturbed full model (the perturbation is introduced in the original full model as part of stability analysis, and is an unknown quantity). If these two reduced systems are equal (\textit{first condition}), with the difference between the original full model and the perturbed full model equal to the order of perturbation  (\textit{second condition}), then the MOR algorithm under consideration is called backward stable. The two theorems summarizing this stability analysis for BIRKA are listed below. 
\begin{theorem}\cite{BIRKAstabilitylaapaper} \label{Theorem:BIRKA_firstcondition}
	If the inexact linear solves in BIRKA (lines 3b. and 3c. of Algorithm \ref{BIRKAAlgo}) are solved using a Petrov-Galerkin framework, then BIRKA satisfies the first
	condition of backward stability with respect to these solves.
\end{theorem}   

\begin{theorem}\cite{BIRKAstabilitylaapaper}\label{Theorem:BIRKA_secondcondition}
	Let  
	$\widehat{Q} = \Biggl( -  \begin{bmatrix}
	A & 0\\ 
	0 & A
	\end{bmatrix}   \otimes    \begin{bmatrix}
	I_n & 0\\ 
	0 & I_n
	\end{bmatrix}     -  \begin{bmatrix}
	I_n & 0\\ 
	0 & I_n
	\end{bmatrix}   \otimes    \begin{bmatrix}
	A & 0\\ 
	0 & A
	\end{bmatrix}    -  \begin{bmatrix}
	N & 0\\ 
	0 & N
	\end{bmatrix}   \otimes    \begin{bmatrix}
	N & 0\\ 
	0 & N
	\end{bmatrix}        \Biggr)$, where $I_n$ is an identity matrix of size $n \times n$ and $\otimes$ denotes the standard Kronecker product. Also, let $\widehat{\widehat{F}} = \left(I_{2n} \otimes \widehat{F} + \widehat{F} \otimes I_{2n} \right)$ with $ \widehat{F} =\begin{bmatrix}
	0 & 0\\ 
	0 & F
	\end{bmatrix} $, where $F$ is the perturbation introduced in $A$ matrix of the input dynamical system and $I_{2n}$ is an identity matrix of size $2n \times 2n$. If $\widehat{Q}$
	is invertible, 	$\left \| \widehat{Q}^{-1} \right \|_2 < 1$, and
	$\left \| \widehat{\widehat{F}} \right \|_2 < 1$, then BIRKA satisfies the second condition of backward stability with
	respect to the inexact linear solves.
\end{theorem}
\par 
Next, we analyze the backward stability of TBIRKA. Here, the \textit{first condition} is satisfied in a way similar to that of BIRKA except that some extra orthogonality conditions are imposed on the linear solver (discussed below). 
\begin{theorem}
	\label{Theorem:TBIRKA_firstcondition}	
	Let the inexact linear solves in TBIRKA (lines 3b. and 3c. of Algorithm \ref{TBIRKAAlgo}) are solved using a Petrov-Galerkin framework with
	\begin{alignat}{2} \label{eq:orthogonality_condition_theorem}
	\sum_{j=1}^{M} W_j \perp R_{b_{j}} \quad \textnormal{and} \quad  \sum_{j=1}^{M} V_j \perp R_{c_{j}}, 
	\end{alignat}
	where $R_{b_{1}}$ and $R_{b_{j}}$ are the residuals in the first equations of 3b. and 3c. of Algorithm \ref{TBIRKAAlgo}, respectively; $R_{c_{1}}$ and $R_{c_{j}}$ are the residuals in the second equations of 3b. and 3c. of Algorithm \ref{TBIRKAAlgo}, respectively; and $j=2, \ \ldots, \ M.$  
	Then, TBIRKA satisfies the first condition of backward stability with respect to these solves.
\end{theorem}
\begin{proof}
	Follows the same pattern as the proof for Theorem 2.1 in \cite{BIRKAstabilitylaapaper}.
\end{proof}	
\par 
From the above theorem, we infer that the underlying iterative solver should \textit{firstly} be based upon a Petrov-Galerkin framework. Since BiConjugate Gradient (i.e., BiCG) is one such algorithm \cite{greenbaum_book,SaaditerativeBook}, we propose its use in TBIRKA. This is exactly same as for BIRKA. \textit{Secondly,} this particular solver should also satisfy the extra orthogonalities \eqref{eq:orthogonality_condition_theorem}.
\par 
These orthogonalities have a form similar to the orthogonalities required while reducing second order linear dynamical systems ((23) and (24) in \cite{NavSISC}), and can be easily satisfied by using a recycling variant of the underlying iterative solver. In \cite{NavSISC}, the ideal iterative solver to be used is Conjugate Gradient (i.e., CG) \cite{greenbaum_book,SaaditerativeBook} (due to the use of Galerkin projection). Hence, to satisfy the similar orthogonalities there, without any extra cost, the authors use Recycling Conjugate Gradient (i.e., RCG) \cite{Michael_RCG}. Since here BiCG is the ideal iterative solver (as discussed above), we propose the use of Recycling BiConjugate Gradient (i.e., RBiCG) \cite{ahuja2012recyclingsiampaper,kahujaphdthesis}, which would ensure that orthogonalities given by \eqref{eq:orthogonality_condition_theorem} are satisfied without any extra cost.
\par 
To satisfy the \textit{second condition} of backward stability of TBIRKA, we need to show that 
\begin{alignat}{2}\label{eq:second_condition_of_backward_stability_truncated_zeta}
\left \| \zeta^M - \widetilde{\zeta}^M \right \|_{H_2} = \mathcal{O} \left ( \left \| F \right \|_2 \right ),
\end{alignat}
where $\zeta^M$ is given by \eqref{eq:bilinear_series_representation_truncated} or 
\begin{subequations}
	\begin{align}
	&\zeta^M = \left\{ H_{1}\left ( s_{1} \right), \ H_{2}\left ( s_{1},\ s_{2} \right),\ \ldots, \ H_{M}\left ( s_{1},\ \ldots, \ s_{M} \right ) \right \},\label{eq:bilinear_series_representation_truncated_using_in_proof_original_system}\\
	\intertext{with $H_{k}\left ( s_{1},\ \ldots,\ s_{k} \right )$ for $k=1,\ \ldots, \ M$ given by \eqref{eq:subsystem_transfer_function} or}
	\begin{split}
	& H_{k}\left ( s_{1},\ \ldots,\ s_{k} \right ) =c\left ( s_{k}I-A \right )^{-1}\\
	&\qquad\qquad\qquad\qquad N \left ( s_{k-1}I-A \right )^{-1}\ \ldots  N\left ( s_{1}I-A \right )^{-1}b,\label{eq:susbsytem_transfer_function_in_proof_original}
	\end{split}
	\end{align}
\end{subequations}
\begin{subequations}
	\begin{align}
	&\widetilde{\zeta}^M = \left\{ \widetilde{H}_{1}\left ( s_{1} \right), \ \widetilde{H}_{2}\left ( s_{1},\ s_{2} \right),\  \ldots, \ \widetilde{H}_{M}\left ( s_{1},\ \ldots, \ s_{M} \right ) \right \},\label{eq:bilinear_series_representation_truncated_using_in_proof_perturbed_system}\\
	\intertext{with for $k=1,\ \ldots, \ M$}
	\begin{split}
	& \widetilde{H}_{k}\left ( s_{1},\ \ldots,\ s_{k} \right ) =c\left ( s_{k}I-\left(A+F\right) \right )^{-1}\\
	& \qquad\qquad\qquad\qquad N \left ( s_{k-1}I-\left(A+F\right) \right )^{-1}\ \ldots  N\left ( s_{1}I-\left(A+F\right) \right )^{-1}b,\label{eq:susbsytem_transfer_function_in_proof_perturbed}
	\end{split}
	\end{align}
\end{subequations}
and assuming perturbation in $A$ matrix of the input dynamical system (as for BIRKA stability; see Theorem \ref{Theorem:BIRKA_secondcondition} earlier).
\par 
One way to satisfy \eqref{eq:second_condition_of_backward_stability_truncated_zeta} is to use the definition of the $H_2-$norm of $\zeta^M-\widetilde{\zeta}^M$, i.e., from Lemma 5.1 of \cite{TBIRKApaper}
\begin{alignat}{4}
\begin{split} \label{eq:H_2_norm_error_system_truncated_Voterra}
& \left \| \zeta^M -\widetilde{\zeta}^M \right \|_{H_2}^{2}   = 
\left ( \begin{bmatrix}
c & -c
\end{bmatrix} \otimes \begin{bmatrix}
c & -c
\end{bmatrix} \right )    \sum_{j=0}^{M} \Biggl[ \biggl ( -  \begin{bmatrix}
A & 0\\ 
0 & A+F
\end{bmatrix}   \otimes          \\ 
& \quad \begin{bmatrix} 
I_n & 0\\ 
0 & I_n
\end{bmatrix} - \begin{bmatrix}
I_n & 0\\ 
0 & I_n
\end{bmatrix}   \otimes    \begin{bmatrix}
A & 0\\ 
0 & A+F
\end{bmatrix}       \biggr )^{-1} \begin{bmatrix}
N & 0 \\ 
0 & N
\end{bmatrix} \otimes \begin{bmatrix}
N & 0 \\ 
0 & N
\end{bmatrix}  \Biggr]^j   \\
& \quad \biggl ( -  \begin{bmatrix}
A & 0\\ 
0 & A+F
\end{bmatrix}   \otimes    \begin{bmatrix} 
I_n & 0\\ 
0 & I_n
\end{bmatrix}     -   \begin{bmatrix}
I_n & 0\\ 
0 & I_n
\end{bmatrix}   \otimes    \begin{bmatrix}
A & 0\\ 
0 & A+F
\end{bmatrix}       \biggr )^{-1}  \\ 
& \quad \left ( \begin{bmatrix}
b\\ 
b
\end{bmatrix} \otimes \begin{bmatrix}
b\\ 
b
\end{bmatrix}  \right ). 
 \end{split}
\end{alignat}
This approach is followed in satisfying the \textit{second condition} of backward stability of BIRKA, but for TBIRKA it turns out to be more challenging. The reason for this is that the definition of the $H_2-$norm of $\zeta-\widetilde{\zeta}$ used in BIRKA is different from \eqref{eq:H_2_norm_error_system_truncated_Voterra}\footnote{Recall, in BIRKA we work with $\zeta$ rather than $\zeta^M$.}, i.e., from Corollary 4.1 of \cite{BIRKApaper} or Theorem 4.5 of \cite{garretphdthesis}
\begin{alignat*}{4}
& \left \| \zeta -\widetilde{\zeta} \right \|_{H_2}^{2}   = 
\left ( \begin{bmatrix}
c & -c
\end{bmatrix} \otimes \begin{bmatrix}
c & -c
\end{bmatrix} \right )    
\biggl ( -  \begin{bmatrix}
A & 0\\ 
0 & A+F
\end{bmatrix}   \otimes    \begin{bmatrix} 
I_n & 0\\ 
0 & I_n
\end{bmatrix}     -     \\
& \qquad  \begin{bmatrix}
I_n & 0\\ 
0 & I_n
\end{bmatrix}  \otimes  \begin{bmatrix}
A & 0\\ 
0 & A+F
\end{bmatrix}    
-   \begin{bmatrix}
N & 0\\ 
0 & N
\end{bmatrix}   \otimes    \begin{bmatrix}
N & 0\\ 
0 & N
\end{bmatrix}          \biggr )^{-1}  
\left ( \begin{bmatrix}
b\\ 
b
\end{bmatrix} \otimes \begin{bmatrix}
b\\ 
b
\end{bmatrix}  \right ). \notag
\end{alignat*}
\par 
Another way to satisfy \eqref{eq:second_condition_of_backward_stability_truncated_zeta} in case of TBIRKA, which turns out to be more easier, is to show that
\begin{alignat}{2}\label{eq:second_condition_of_backward_stability_series}
\begin{split}
\left \| H_{1}\left ( s_{1} \right ) - \widetilde{H}_{1}\left ( s_{1} \right )  \right \|_{H_{2}} & \propto  \mathcal{O}\left ( \left \| F \right \|_2 \right ), \\
\left \| H_{2}\left ( s_{1},\ s_{2} \right ) - \widetilde{H}_{2}\left ( s_{1},\ s_{2} \right )  \right \|_{H_{2}} &\propto  \mathcal{O}\left ( \left \| F \right \|_2 \right ), \\
& \vdots\\
\left \| H_{M}\left ( s_{1},\ \ldots,\ s_{M} \right ) - \widetilde{H}_{M}\left ( s_{1},\ \ldots,\ s_{M} \right )  \right \|_{H_{2}} &\propto  \mathcal{O}\left ( \left \| F \right \|_2 \right ),
\end{split}
\end{alignat}
where $ H_{k}\left ( s_{1},\ \ldots,\ s_{k} \right )$ for $k=1, \ \ldots,\ M$ is given by \eqref{eq:subsystem_transfer_function} and \eqref{eq:susbsytem_transfer_function_in_proof_original}, and $\widetilde{H}_{k}\left ( s_{1},\ \ldots,\ s_{k} \right )$ for $k=1, \ \ldots,\ M$ is given by \eqref{eq:susbsytem_transfer_function_in_proof_perturbed}. This way was not possible in BIRKA because there $M \rightarrow \infty$ (see \eqref{eq:originalbilinearequation}-\eqref{eq:subsystem_transfer_function}).
\section{Satisfying the Second Condition of Backward Stability} \label{sec:Stability_Analysis}
\par 
We use induction to prove \eqref{eq:second_condition_of_backward_stability_series}.
To prove this condition, we first abstract out the term containing the perturbation $F$ from the normed difference between the transfer function of the original system $\left(  H_{M}\left ( s_{1},\ \ldots,\ s_{M} \right ) \right)$ and the transfer function of the perturbed system  $\left( \widetilde{H}_{M}\left ( s_{1},\ \ldots,\ s_{M} \right ) \right)$ (in Lemma \ref{Lemma:Mth_term_H2_norm}). Next, we use induction to prove that this abstracted term is $\mathcal{O}\left ( \left \| F \right \|_2 \right )$ (in Lemma \ref{Lemma:induction_proof}). Finally, we use the result of these two lemmas to prove \eqref{eq:second_condition_of_backward_stability_series} (in Theorem \ref{Theorem:second_condition_subsystem}). Note, that in all our subsequent derivations, we assume that all inverses used exist. This is an 
acceptable assumption because the inverse of matrices arising here are of the form as in \cite{sarahinexactlaapaper} and \cite{BIRKAstabilitylaapaper} (the papers that discuss stability of IRKA and BIRKA, respectively).
\begin{lemma} \label{Lemma:Mth_term_H2_norm}
If 
\begin{align*}
H_{M}\left ( s_{1},\ \ldots,\ s_{M} \right ) =c\left ( s_{M}I_n-A \right )^{-1}N \left ( s_{M-1}I_n-A \right )^{-1}      \ldots N\left ( s_{1}I_n-A \right )^{-1}b
\end{align*}
 and 
 \begin{align*}
 \widetilde{H}_{M}\left ( s_{1},\ \ldots,\ s_{M} \right )  = & c\left ( s_{M}I_n-\left ( A+F \right ) \right )^{-1} \\
 & \ N \left ( s_{M-1}I_n- \left ( A+F \right ) \right )^{-1} \ldots   N\left ( s_{1}I_n- \left ( A+F \right ) \right )^{-1}b,
 \end{align*} then the $H_2-$norm of their difference 
	\begin{alignat*}{3} 
	&\left \| H_{M}\left ( s_{1},\ \ldots,\ s_{M} \right ) - \widetilde{H}_{M}\left ( s_{1},\ \ldots,\ s_{M} \right )  \right \|_{H_{2}}^{2} \leq   \left \| c \mathcal{K}^{-1}\left(s_{M}\right) \right \|_{H_{2}}^{2} \\ 
	 &\quad \left \| \mathcal{K}^{-1}\left(s_{M-1}\right) \right \|_{H_{2}}^{2}  \ldots   \left \| \mathcal{K}^{-1}\left(s_{1}\right) \right \|_{H_{2}}^{2} \left \| U(s_1,\ \ldots,\ s_M) \right \|_{H_{\infty}}^{2}  \left \| \mathcal{K}^{-1}\left(s_1\right)b \right \|_{H_{\infty}}^{2}
	\end{alignat*}
	where, $	\mathcal{K}\left(s_i\right)=\left ( s_{i}I_{n}-A \right ) $ for $i= 1, \ \ldots, \ M$ and 
	\begin{alignat}{2} \label{eq:U_expression}
	\begin{split}
	& U(s_1, \ldots,\ s_{M}) = \mathcal{K}\left ( s_1 \right )\ldots \mathcal{K}\left ( s_{M-1} \right ) 
	\Bigg( N \mathcal{K}^{-1}\left(s_{M-1}\right) \ldots N\mathcal{K}^{-1}\left(s_2\right) N   \\
	&\qquad -\left ( I_n- F \mathcal{K}^{-1}\left(s_M\right) \right )^{-1}   N \mathcal{K}^{-1}\left(s_{M-1}\right) \left ( I_n- F \mathcal{K}^{-1}\left(s_{M-1}\right) \right )^{-1} \ldots \\
	&\qquad\quad   N \mathcal{K}^{-1}\left(s_2\right) \left ( I_n- F \mathcal{K}^{-1}\left(s_2\right) \right )^{-1}  N  \left ( I_n- \mathcal{K}^{-1}\left(s_1\right) F\right )^{-1}\Bigg).
	\end{split}
	\end{alignat}
\end{lemma}	
\begin{proof}
	Using the definition of $H_2-$norm \eqref{eq:H_2normdefinitionsubsystem}, we get
	\begin{alignat}{2}
	& \left \| H_{M}\left ( s_{1},\ \ldots,\ s_{M} \right ) - \widetilde{H}_{M}\left ( s_{1},\ \ldots,\ s_{M} \right )  \right \|_{H_{2}}^{2}  \notag\\
	& = \left(\frac{1}{2\pi }\right)^M  \underset{m\rightarrow \infty}{Lim} \int_{-m}^{m} \ldots \int_{-m}^{m}  
	\left \|c \mathcal{K}^{-1}\left(i\omega_M\right) N \mathcal{K}^{-1}\left(i\omega_{M-1}\right)\ldots  N\mathcal{K}^{-1}\left(i\omega_1\right)b  \ 
	\right.\notag\\
	&  \quad - c \left(\mathcal{K}\left ( i\omega_{M} \right)-F\right )^{-1}  N \left(\mathcal{K}\left ( i\omega_{M-1} \right)-F\right )^{-1} \ldots N \left(\mathcal{K}\left ( i\omega_{2} \right)-F\right )^{-1}  \notag\\
	& \qquad \left.  N \left(\mathcal{K}\left ( i\omega_{1}\right)-F\right )^{-1}b \right \|_{F}^{2}   d\omega_{1} \ldots  d\omega_{M} \notag
	\end{alignat}
	\begin{alignat}{2}
	& = \left(\frac{1}{2\pi }\right)^M \underset{m\rightarrow \infty}{Lim} \int_{-m}^{m} \ldots \int_{-m}^{m}   \left \| c \mathcal{K}^{-1}\left(i\omega_M\right)\Bigg(
	N \mathcal{K}^{-1}\left(i\omega_{M-1}\right)\ldots  N \mathcal{K}^{-1}\left(i\omega_2\right)  N  \ 
	\right.\notag\\
	&  \quad  -\left(I_n-F \mathcal{K}^{-1}\left(i\omega_M\right)\right )^{-1}  N \mathcal{K}^{-1}\left(i\omega_{M-1}\right)\left(I_n-F \mathcal{K}^{-1}\left(i\omega_{M-1}\right)\right )^{-1} \ldots\notag\\
	& \qquad \left.  N \mathcal{K}^{-1}\left(i\omega_2\right)\left(I_n-F \mathcal{K}^{-1}\left(i\omega_2\right)\right )^{-1} 
	 N \left(I_n- \mathcal{K}^{-1}\left(i\omega_1\right)F\right )^{-1}\Bigg) \mathcal{K}^{-1}\left(i\omega_1\right) b   \right \|_{F}^{2}  \notag \\
	 & \qquad d\omega_{1} \ldots  d\omega_{M} \notag\\
	\begin{split}\notag
	& = \left(\frac{1}{2\pi }\right)^M \underset{m\rightarrow \infty}{Lim} \int_{-m}^{m} \ldots \int_{-m}^{m}   \left \| c \mathcal{K}^{-1}\left(i\omega_M\right) \mathcal{K}^{-1}\left(i\omega_{M-1}\right) \ldots \mathcal{K}^{-1}\left(i\omega_1\right) \right.\\
	&\qquad \mathcal{K}\left(i\omega_{1}\right)\ldots \mathcal{K}\left(i\omega_{M-1}\right) \Bigg(
	N \mathcal{K}^{-1}\left(i\omega_{M-1}\right)\ldots  N \mathcal{K}^{-1}\left(i\omega_2\right)  N \ 
	\\
	& \quad - \left(I_n-F \mathcal{K}^{-1}\left(i\omega_M\right)\right )^{-1}  N \mathcal{K}^{-1}\left(i\omega_{M-1}\right)\left(I_n-F \mathcal{K}^{-1}\left(i\omega_{M-1}\right)\right )^{-1} \ldots \\
	& \qquad \left. N \mathcal{K}^{-1}\left(i\omega_2\right)\left(I_n-F \mathcal{K}^{-1}\left(i\omega_2\right)\right )^{-1}
	 N \left(I_n- \mathcal{K}^{-1}\left(i\omega_1\right)F\right )^{-1}\Bigg) \mathcal{K}^{-1}\left(i\omega_1\right) b   \right \|_{F}^{2}\\
	 & \qquad    d\omega_{1} \ldots  d\omega_{M}. 		 
	\end{split}
	\end{alignat}
	Using $U\left(s_1, \ \ldots, \ s_M\right)$ given by \eqref{eq:U_expression}, $\left \| X Y Z \right \|_{F}   \leq \left \| X \right \|_{F}  \left \| Y Z \right \|_{F}$, $\left \|  Y Z \right \|_{F}   \leq \left \| Y \right \|_{F}  \left \| Z \right \|_{2}$, and comparison integral inequality\footnote{\label{footnote1}This inequality says if $f\left(x\right)$ and $g\left(x\right)$ are integrable over $\left[a,b\right]$ and $f\left(x\right) \leq g\left(x\right)$, then $\int_{a}^{b} f \left ( x \right ) dx \leq \int_{a}^{b} g \left ( x \right ) dx$. Note that although we have improper integrals here, this inequality still holds because of the earlier assumption that such integrals give a finite value.} \cite{Alen_Jeffrey__mean_value_theorem_book} for any matrices $X,\ Y$ and $Z$, in the above equation, we have 
	\begin{alignat}{2}\label{eq:Mth_error_term_H2_norm_intermediate_limit}
	\begin{split}
	& \left \| H_{M}\left ( s_{1},\ \ldots,\ s_{M} \right ) - \widetilde{H}_{M}\left ( s_{1},\ \ldots,\ s_{M} \right )  \right \|_{H_{2}}^{2} \\
	&\quad \leq  \left(\frac{1}{2\pi }\right)^M \underset{m\rightarrow \infty}{Lim}  
	\int_{-m}^{m} \ldots \int_{-m}^{m}   \left \| c \mathcal{K}^{-1}\left(i\omega_M\right)  \right\|_F^2  \left \|  \mathcal{K}^{-1}\left(i\omega_{M-1}\right)  \right\|_F^2 \ldots  \\ 
	& \qquad  \left \|  \mathcal{K}^{-1}\left(i\omega_1\right) \right\|_F^2 \left \| U \left(i\omega_{1}, \ \ldots, \ i \omega_{M}\right) \right\|_2^2  \left \|  \mathcal{K}^{-1}\left(i\omega_1\right) b \right\|_2^2    d\omega_{1} \ldots d\omega_{M}.  	
	\end{split}
	\end{alignat}
	From the mean value theorem of integration \cite{Alen_Jeffrey__mean_value_theorem_book} we know
	\begin{alignat}{4}
	\int_{-m}^{m} \int_{-m}^{m}  f\left ( i\omega_{2} \right ) & g\left ( i\omega _{1},i\omega _{2} \right ) h\left ( i\omega _{1} \right )  d\omega _{1}d\omega _{2} \notag\\
	= & \int_{-m}^{m}  f\left ( i\omega_{2} \right )   \left ( \int_{-m}^{m}  g\left ( i\omega _{1},i\omega _{2} \right ) h\left ( i\omega _{1} \right )d\omega _{1}  \right)   d\omega _{2}\notag
	\end{alignat}
	\begin{alignat}{2}
	\leq & \int_{-m}^{m}  f\left ( i\omega_{2} \right )   \left (  \underset{c\in \mathbb{R}}{max}\left ( g(ic,i\omega _{2}) \right ) \int_{-m}^{m}   h\left ( i\omega _{1} \right )d\omega _{1}  \right)   d\omega _{2} \notag\\
	\leq & \left (    \int_{m}^{m}  f\left ( i\omega_{2} \right )     \underset{c\in \mathbb{R}}{max}\left ( g(ic,i\omega _{2}) \right ) d\omega _{2} \right ) \int_{-m}^{m}   h\left ( i\omega _{1} \right )d\omega _{1}   \notag\\
	\leq & \ \underset{c,d\in \mathbb{R}}{max} (g(ic,id))   \int_{-m}^{m}  f\left ( i\omega_{2} \right )     d\omega _{2} \int_{-m}^{m}   h\left ( i\omega _{1} \right )d\omega _{1}.  \notag
	\end{alignat}
	Using this property in \eqref{eq:Mth_error_term_H2_norm_intermediate_limit} we get\footnote{As mentioned in Footnote \ref{footnote1}, the improper integral does not affect application of this mean value theorem because all such integrals are assumed to give a finite value. }
	\begin{alignat*}{2}
	& \left \| H_{M}\left ( s_{1},\ \ldots,\ s_{M} \right ) - \widetilde{H}_{M}\left ( s_{1},\ \ldots,\ s_{M} \right )  \right \|_{H_{2}}^{2} \\
	&\qquad \leq  \left(\frac{1}{2\pi }\right)^M \underset{m\rightarrow \infty}{Lim}  
	\int_{-m}^{m} \ldots \int_{-m}^{m}   \left \| c \mathcal{K}^{-1}\left(i\omega_M\right)  \right\|_F^2  \left \|  \mathcal{K}^{-1}\left(i\omega_{M-1}\right)  \right\|_F^2 \ldots\\
	& \qquad\qquad  \left \|  \mathcal{K}^{-1}\left(i\omega_1\right) \right\|_F^2  d\omega_{1} \ldots d\omega_{M}  \underset{\omega_1,\ \ldots,\ \omega_M \in \mathbb{R}}{max} \left \| U \left(i\omega_{1}, \ \ldots, \ i \omega_{M}\right) \right\|_2^2 \ \\
	&\qquad\qquad  \underset{\omega_1 \in \mathbb{R}}{max} \left \|  \mathcal{K}^{-1}\left(i\omega_1\right) b \right\|_2^2  \\
	&\qquad \leq   \left \| c \mathcal{K}^{-1}\left(s_M\right)  \right\|_{H_2}^2 \left \|  \mathcal{K}^{-1}\left(s_{M-1}\right)  \right\|_{H_2}^2 \ldots \left \|  \mathcal{K}^{-1}\left(s_1\right)  \right\|_{H_2}^2 \\
	& \qquad\qquad  \left \| U \left(s_{1}, \ \ldots, \ s_{M}\right) \right\|_{H_{\infty}}^2  \left \| \mathcal{K}^{-1}\left(s_1\right) b \right\|_{H_{\infty}}^2.  	
	\end{alignat*}
\end{proof}
\begin{lemma} \label{Lemma:induction_proof}
	If $U_{M} = U(s_1,\ \ldots,\ s_{M}) $ defined in \eqref{eq:U_expression}, $\left \| F \right \|_2 < 1 $, and\\ $\left \| \mathcal{K}^{-1} \left ( s \right ) \right \|_{H_{\infty}}<1$, where $\mathcal{K}\left(s\right) = \left(s I_n - A\right)$.
	Then we have
	\begin{align*}
	\left\| U_{M} \right\|_{H_{\infty}}  \propto \mathcal{O} \left( \left\| F \right\|_2  \right).
	\end{align*}
\end{lemma}
\begin{proof}
	We prove this by mathematical induction.\\
	\underline{\textbf{Base Case}}
	\begin{enumerate} [label=(\roman*)]
		\item \textit{First subsystem}\\
		This is the linear system case, which has been already proved in \cite{sarahinexactlaapaper} (see Theorem 4.3 of \cite{sarahinexactlaapaper}).
		\item \textit{Second subsystem} \\
		Substituting $M=2$ in \eqref{eq:U_expression}, we get
		\begin{align*}	  
		U_2  = \mathcal{K}\left(s_1\right)\left(N- \left ( I_n-F\mathcal{K}^{-1} \left(s_2\right) \right )^{-1} N \left ( I_n- \mathcal{K}^{-1} \left(s_1\right)  F \right )^{-1}  \right).   
		\end{align*}
		If $ \left\| F\mathcal{K}^{-1} \left(s_2\right)\right\|_{H_{\infty}} < 1$ and $ \left\| \mathcal{K}^{-1} \left(s_1\right) F\right\|_{H_{\infty}}<1$, then by the Neumann series, we get\footnote{From \cite[page 527]{Matrix_Analysis_carl_meyer_book}, we know $\left ( I - A \right )^{-1} = \sum\limits_{k=0}^{\infty} A^k$ when $ \left \| A \right \|<1$ for any matrix norm. Here, for the first inequality we have $\left \| F \mathcal{K}^{-1} \left ( s_2 \right ) \right \|_{H_{\infty}} <1$ or $\underset{\omega_2 \in \mathbb{R}}{max}  \left \| F \mathcal{K}^{-1} \left ( i\omega_2 \right ) \right \|_2 <1$, and hence, the applicable matrix norm is $2-$norm. Similarly for the second inequality.} 
		\begin{alignat*}{2}	  
		U_2  &= \mathcal{K}\left(s_1\right)\bigg(N- \left ( I_n+F\mathcal{K}^{-1}\left(s_2\right) + \left(F\mathcal{K}^{-1}\left(s_2\right)\right)^2 +\cdots \right ) N \\
		&\quad \left ( I_n+ \mathcal{K}^{-1} \left(s_1\right)  F + \left(\mathcal{K}^{-1} \left(s_1\right)  F\right)^2 +\cdots \right )  \bigg)\\
		&= \mathcal{K}\left(s_1\right)\bigg(N- N-N\mathcal{K}^{-1} \left(s_1\right)F\left ( I_n+ \mathcal{K}^{-1} \left(s_1\right)  F + \cdots \right ) - \\
		& \quad F \mathcal{K}^{-1} \left(s_2\right) \left ( I_n+F\mathcal{K}^{-1}\left(s_2\right) + \cdots \right ) N \\
		& \quad  \left ( I_n+ \mathcal{K}^{-1} \left(s_1\right)  F + \left(\mathcal{K}^{-1} \left(s_1\right)  F \right)^2 + \cdots \right )  \bigg)\\
		&=\mathcal{K}\left(s_1\right)\bigg(-N\mathcal{K}^{-1} \left(s_1\right)F\left ( I_n- \mathcal{K}^{-1} \left(s_1\right)  F \right )^{-1} -\\
		& \quad F \mathcal{K}^{-1} \left(s_2\right) \left ( I_n-F\mathcal{K}^{-1}\left(s_2\right) \right )^{-1} N \left ( I_n- \mathcal{K}^{-1} \left(s_1\right)  F  \right )^{-1}  \bigg)\\
		&=\mathcal{K}\left(s_1\right)\Big(-N\mathcal{K}^{-1} \left(s_1\right)F - F \mathcal{K}^{-1} \left(s_2\right) \left ( I_n-F\mathcal{K}^{-1}\left(s_2\right) \right )^{-1}  N   \Big)\\
		& \quad \left ( I_n- \mathcal{K}^{-1} \left(s_1\right)  F  \right )^{-1}.
		\end{alignat*}
		Taking $H_{\infty}-$norm on both sides, we get
		\begin{alignat*}{2}
		\left \| U_2 \right \|_{H_{\infty}} &=  \left \| \mathcal{K}\left(s_1\right)\Big(-N\mathcal{K}^{-1} \left(s_1\right)F - F \mathcal{K}^{-1} \left(s_2\right) \left ( I_n-F\mathcal{K}^{-1}\left(s_2\right) \right )^{-1}  N   \Big) \right.\\
		& \qquad \left. \left ( I_n- \mathcal{K}^{-1} \left(s_1\right)  F  \right )^{-1} \right \|_{H_{\infty}}\\
		&=  \underset{\omega_1, \omega_2 \in \mathbb{R}}{max} \left \| \mathcal{K}\left(i\omega_1\right)\Big(-N\mathcal{K}^{-1} \left(i\omega_1\right)F - F \mathcal{K}^{-1} \left(i\omega_2\right)\right.\\
		&\qquad \left. \left ( I_n-F\mathcal{K}^{-1}\left(i\omega_2\right) \right )^{-1}  N   \Big)   \left ( I_n- \mathcal{K}^{-1} \left(i\omega_1\right)  F  \right )^{-1} \right \|_2.
		\end{alignat*}
		Using $\left\| X Y\right\|_2 \leq \left\| X\right\|_2 \left\| Y\right\|_2$ and $\left\| X+Y\right\|_2 \leq \left\| X\right\|_2 + \left\| Y\right\|_2$, for any two matrices $X$ and $Y$, in the above equation, we get
		\begin{align}
		\left \| U_2 \right \|_{H_{\infty}} 
		\leq & \underset{\omega_1, \omega_2 \in \mathbb{R}}{max} \Bigg(\left \| \mathcal{K}\left(i\omega_1\right)\right \|_2 \bigg(  \left \| N \right \|_2 \left \|\mathcal{K}^{-1} \left(i\omega_1\right)\right \|_2 	\left \| F \right \|_2 + \notag\\
		&  \left \| F \right \|_2 \left \| \mathcal{K}^{-1} \left(i\omega_2\right) \right \|_2  \left \| \left ( I_n-F\mathcal{K}^{-1}\left(i\omega_2\right) \right )^{-1}     \right \|_2 \left \| N \right \|_2 \bigg) \notag\\
		& \left \|\left ( I_n- \mathcal{K}^{-1} \left(i\omega_1\right)  F  \right )^{-1} \right \|_2 \Bigg) \notag\\
		\begin{split} \label{eq:U_2_expression_inequality_max_inside}
		\leq &\left \| \mathcal{K}\left(s_1\right)\right \|_{H_{\infty}}  \left \| N \right \|_{2} 	\left \| F \right \|_{2}  \Bigg(  \left \|\mathcal{K}^{-1} \left(s_1\right)\right \|_{H_{\infty}}   + \\
		& \left \| \mathcal{K}^{-1} \left(s_2\right) \right \|_{H_{\infty}} \ \underset{\omega_2 \in \mathbb{R}}{max}  \left \| \left ( I_n-F\mathcal{K}^{-1}\left(i\omega_2\right) \right )^{-1}     \right \|_2  \Bigg) \\
		& \underset{\omega_1 \in \mathbb{R}}{max}  \left \|\left ( I_n- \mathcal{K}^{-1} \left(i\omega_1\right)  F  \right )^{-1} \right \|_2.
		\end{split}
		\end{align}
		Technically by definition of the $H_{\infty}-$norm and how $\mathcal{K} \left(s\right)$ is defined in our hypotheses, $\left\| \mathcal{K} \left(s_1\right)\right\|_{H_{\infty}} = \left\| \mathcal{K} \left(s_2\right)\right\|_{H_{\infty}} = \left\| \mathcal{K} \left(s\right)\right\|_{H_{\infty}}$, however, for sake of exposition, we keep them separate. Similarly for the $H_{\infty}-$norm of inverses of $\mathcal{K} \left(s_1\right)$ and $\mathcal{K} \left(s_2\right)$.
		\par
		\sloppy
		\quad To abstract $\left\| F\right\|_2$ out from the above inequality, let us look at  $\underset{\omega_2 \in \mathbb{R}}{max}  \left \| \left ( I_n-F\mathcal{K}^{-1}\left(i\omega_2\right) \right )^{-1} \right\|_2$ separately. Recall, while applying Neumann series we assumed that $\left \| F \mathcal{K}^{-1} \left ( s_2 \right ) \right \|_{H_{\infty}} <1$ or  $ \underset{\omega_2 \in \mathbb{R}}{max} \left \| F \mathcal{K}^{-1} \left ( i\omega_2 \right ) \right \|_2 <1$. Since the maximum of such a norm is less than one, we have for all $\omega_2 \in \mathbb{R}$,  $\left \| F \mathcal{K}^{-1} \left ( i\omega_2 \right ) \right \|_2 <1$. Using this along with Lemma 2.3.3 from \cite{g_golub_book}\footnote{If $F \in \mathbb{R}^{n \times n}$ and $\left \| F \right \|_p <1$, then $I-F$ is nonsingular and $\left ( I-F \right )^{-1} =\sum\limits_{k=0}^{\infty} F^k $ with $	\left \| \left ( I-F \right )^{-1} \right \|_p \leq \dfrac{1}{1- \left \| F \right \|_p}$.} in the above expression, we get
		\begin{alignat}{3}
		\underset{\omega_2 \in \mathbb{R}}{max}  \left \| \left ( I_n-F\mathcal{K}^{-1}\left(i\omega_2\right) \right )^{-1} \right\|_2 &\leq  \underset{\omega_2 \in \mathbb{R}}{max} \ \frac{1}{1- \left \| F \mathcal{K}^{-1} \left ( i\omega_2 \right ) \right \|_2} \notag\\
		&\leq  \frac{1}{1-  \underset{\omega_2 \in \mathbb{R}}{max}  \left \| F \mathcal{K}^{-1} \left ( i\omega_2 \right ) \right \|_2} \notag\\
		& \leq  \frac{1}{1-  \left \| F \mathcal{K}^{-1} \left ( s_2 \right ) \right \|_{H_{\infty}}}. \label{eq:max_to_H_infty_inequality_golub_s2}
		\end{alignat}
		If we assume $\left \| F \right \|_2 < 1$ and $\left \| \mathcal{K}^{-1} \left(s\right) \right \|_{H_{\infty}} < 1$ (as in our hypotheses), then using earlier used matrix norm properties, we get
		\begin{alignat*}{2}
		\left\|F\mathcal{K}^{-1}\left(s_2\right)\right\|_{H_{\infty}} = \underset{\omega_2 \in \mathbb{R}}{max} \left\|F\mathcal{K}^{-1}\left(i\omega_2\right)\right\|_2  & \leq  \left \|F\right \|_2 \underset{\omega_2 \in \mathbb{R}}{max} \left\|\mathcal{K}^{-1}\left(i\omega_2\right)\right\|_2 \\
		&\leq \left \|F\right \|_2 \left\|\mathcal{K}^{-1}\left(s_2\right)\right\|_{H_{\infty}} \\
		&\leq 1,
		\end{alignat*}
		as assumed for applying Neumann series earlier as well as Lemma 2.3.3 from \cite{g_golub_book} above. Thus, no extra assumptions beyond those in hypotheses are needed. Further, we also get 
		\begin{alignat}{2}
		1- \left \|F\right \|_2 \left\|\mathcal{K}^{-1}\left(s_2\right)\right\|_{H_{\infty}} & \leq  1-  \left\|F \mathcal{K}^{-1}\left(s_2\right)\right\|_{H_{\infty}} \qquad \textnormal{or} \notag\\
		\frac{1}{ 1-  \left\|F \mathcal{K}^{-1}\left(s_2\right)\right\|_{H_{\infty}}} & \leq \frac{1}{1- \left \|F\right \|_2 \left\|\mathcal{K}^{-1}\left(s_2\right)\right\|_{H_{\infty}} }. \label{eq:golub_inequality_rational_form_individual_2_infty_norm_for_s2}
		\end{alignat}
		\par 
		\quad 
		Similarly, we can bound the last term of \eqref{eq:U_2_expression_inequality_max_inside} as follows:
		\begin{alignat}{3}
		\underset{\omega_1 \in \mathbb{R}}{max}  \left \| \left ( I_n-\mathcal{K}^{-1}\left(i\omega_1\right)F \right )^{-1} \right\|_2 &\leq   \frac{1}{1- \left \|  \mathcal{K}^{-1} \left ( s_1 \right ) F \right \|_{H_{\infty}} } \qquad \textnormal{and} \label{eq:max_to_H_infty_inequality_golub_s1}\\
		\frac{1}{ 1-  \left\| \mathcal{K}^{-1}\left(s_1\right)F\right\|_{H_{\infty}}} & \leq \frac{1}{1-  \left\|\mathcal{K}^{-1}\left(s_1\right)\right\|_{H_{\infty}} \left \|F\right \|_2 }. \label{eq:golub_inequality_rational_form_individual_2_infty_norm_for_s1}
		\end{alignat}
		\par 
		\quad
		Substituting \eqref{eq:max_to_H_infty_inequality_golub_s2}-\eqref{eq:golub_inequality_rational_form_individual_2_infty_norm_for_s2} and \eqref{eq:max_to_H_infty_inequality_golub_s1}-\eqref{eq:golub_inequality_rational_form_individual_2_infty_norm_for_s1} in \eqref{eq:U_2_expression_inequality_max_inside}, we get
		\begin{align*}
		\left \| U_2 \right \|_{H_{\infty}} 
		\leq &\left \| \mathcal{K}\left(s_1\right)\right \|_{H_{\infty}}  \left \| N \right \|_{2} 	\left \| F \right \|_{2}  \Bigg[  \left \|\mathcal{K}^{-1} \left(s_1\right)\right \|_{H_{\infty}}   + \\
		& \dfrac{\left \| \mathcal{K}^{-1} \left(s_2\right) \right \|_{H_{\infty}}}{1-\left \|  F \right \|_{2} \left \|  \mathcal{K}^{-1}\left(s_2\right) \right \|_{H_{\infty}}}     
		\Bigg] \left( \dfrac{1}{1-\left \| \mathcal{K}^{-1} \left(s_1\right)    \right \|_{H_{\infty}} \left \|   F  \right \|_{2}} \right).
		\end{align*}
		From the above inequality it is clear that if  $\left\|F \right\|_{2} \left\|\mathcal{K}^{-1} \left(s_2\right) \right\|_{H_{\infty}} < 1$  and $\left\| \mathcal{K}^{-1} \left(s_1\right) \right\|_{H_{\infty}} \left\| F \right\|_{2}< 1$, which are true from our hypotheses, then
		\begin{equation*}
		\left \| U_2 \right \|_{H_{\infty}}  =  \mathcal{O} \left ( \left \| F \right \|_2 \right ). 
		\end{equation*}
		
		\item \textit{Third subsystem}\\
		Using $M=3$ in \eqref{eq:U_expression}, we get
		\begin{alignat*}{2}
		U_3  = \ &
		\mathcal{K}\left(s_1\right) \mathcal{K}\left(s_2\right)\Big( N \mathcal{K}^{-1} \left(s_2\right) N - \left ( I_n-F\mathcal{K}^{-1} \left(s_3\right) \right )^{-1} N \mathcal{K}^{-1} \left(s_2\right) \\ & \left ( I_n-F\mathcal{K}^{-1} \left(s_2\right) \right )^{-1} N\left ( I_n- \mathcal{K}^{-1} \left(s_1\right)  F \right )^{-1}  \Big).   
		\end{alignat*}
		Following the same steps as in the case of the second subsystem we get 
		\begin{align*}
		\left \| U_3 \right \|_{H_{\infty}} 
		\leq &\left \| \mathcal{K}\left(s_1\right)\right \|_{H_{\infty}} 
		\left \| \mathcal{K}\left(s_2\right)\right \|_{H_{\infty}} \left \| N \right \|_{2}^2 	\left \| F \right \|_{2}  \left \| \mathcal{K}^{-1} \left(s_2\right)\right \|_{H_{\infty}}  \\
		& \Bigg[  \left \|\mathcal{K}^{-1} \left(s_1\right)\right \|_{H_{\infty}}   +  \dfrac{\left \| \mathcal{K}^{-1} \left(s_2\right) \right \|_{H_{\infty}}}{1-\left \|  F \right \|_{2} \left \|  \mathcal{K}^{-1}\left(s_2\right) \right \|_{H_{\infty}}}  
		+ \\
		& \dfrac{\left \| \mathcal{K}^{-1} \left(s_3\right) \right \|_{H_{\infty}}}{\left(1-\left \|  F \right \|_{2} \left \|  \mathcal{K}^{-1}\left(s_3\right) \right \|_{H_{\infty}}\right) \left(1-\left \|  F \right \|_{2} \left \|  \mathcal{K}^{-1}\left(s_2\right) \right \|_{H_{\infty}}\right)}  
		\Bigg] \\& \left( \dfrac{1}{1-\left \| \mathcal{K}^{-1} \left(s_1\right)    \right \|_{H_{\infty}} \left \|   F  \right \|_{2}} \right).
		\end{align*}
		Again, from the above inequality it is clear that if $\left\|F \right\|_{2} \left\|\mathcal{K}^{-1} \left(s_2\right) \right\|_{H_{\infty}}$, $\left\|F \right\|_{2} \left\|\mathcal{K}^{-1} \left(s_3\right) \right\|_{H_{\infty}}$ and  $\left\| \mathcal{K}^{-1} \left(s_1\right) \right\|_{H_{\infty}} \left\| F \right\|_{2}$ are all less than $1$, which are true from our hypotheses, then
		\begin{align*}
		\left \| U_3 \right \|_{H_{\infty}}  =  \mathcal{O} \left ( \left \| F \right \|_2 \right ). 
		\end{align*}
	\end{enumerate}
	\underline{\textbf{Induction Hypothesis}}\\
	From \eqref{eq:U_expression}, we know
	\begin{alignat}{2}\label{eq:U_expression_induction}
	\begin{split}
	U_M &= \mathcal{K}\left ( s_1 \right )\ldots \mathcal{K}\left ( s_{M-1} \right ) 
	\Bigg( N \mathcal{K}^{-1}\left(s_{M-1}\right) \ldots N\mathcal{K}^{-1}\left(s_2\right) N   \\
	&\quad -\left ( I_n- F \mathcal{K}^{-1}\left(s_M\right) \right )^{-1}   N \mathcal{K}^{-1}\left(s_{M-1}\right) \left ( I_n- F \mathcal{K}^{-1}\left(s_{M-1}\right) \right )^{-1} \ldots \\
	& \qquad \quad N \mathcal{K}^{-1}\left(s_2\right) \left ( I_n- F \mathcal{K}^{-1}\left(s_2\right) \right )^{-1}  N  \left ( I_n- \mathcal{K}^{-1}\left(s_1\right) F\right )^{-1}\Bigg).
	\end{split}
	\end{alignat}
	Let
	\begin{align*}
	\left \| U_M \right \|_{H_{\infty}} 
	\leq &\left \| \mathcal{K}\left(s_1\right)\right \|_{H_{\infty}} 
	\ldots \left \| \mathcal{K}\left(s_{M-1}\right)\right \|_{H_{\infty}} \left \| N \right \|_{2}^{M-1} 	\left \| F \right \|_{2}  \left \| \mathcal{K}^{-1}\left(s_{M-1}\right)\right \|_{H_{\infty}}    \\
	& \left \| \mathcal{K}^{-1} \left(s_2\right)\right \|_{H_{\infty}} \Bigg[  \left \|\mathcal{K}^{-1} \left(s_1\right)\right \|_{H_{\infty}}   +  \dfrac{\left \| \mathcal{K}^{-1} \left(s_2\right) \right \|_{H_{\infty}}}{1-\left \|  F \right \|_{2} \left \|  \mathcal{K}^{-1}\left(s_2\right) \right \|_{H_{\infty}}}  \\
	& \quad  \vdots \\
	&+  \dfrac{\left \| \mathcal{K}^{-1} \left(s_M\right) \right \|_{H_{\infty}}}{\left(1-\left \|  F \right \|_{2} \left \|  \mathcal{K}^{-1}\left(s_M\right) \right \|_{H_{\infty}}\right) \ldots \left(1-\left \|  F \right \|_{2} \left \|  \mathcal{K}^{-1}\left(s_2\right) \right \|_{H_{\infty}}\right)} 
	\Bigg] \\
	& \left( \dfrac{1}{1-\left \| \mathcal{K}^{-1} \left(s_1\right)    \right \|_{H_{\infty}} \left \|   F  \right \|_{2}} \right) \quad \textnormal{or}\\
	\left \| U_{M} \right \|_{H_{\infty}}  =& \ \mathcal{O} \left ( \left \| F \right \|_2 \right ).
	\end{align*}
	\underline{\textbf{Induction Step}}\\
	Again, from \eqref{eq:U_expression}, we know 
	\begin{alignat*}{2}
	\begin{split}
	U_{M+1} &= \mathcal{K}\left ( s_1 \right )\ldots \mathcal{K}\left(s_M\right) 
	\Bigg( N \mathcal{K}^{-1}\left(s_{M}\right) \ldots N\mathcal{K}^{-1}\left(s_2\right) N   \\
	&\quad -\left ( I_n- F \mathcal{K}^{-1}\left(s_{M+1}\right) \right )^{-1}   N \mathcal{K}^{-1}\left(s_{M}\right) \left ( I_n- F \mathcal{K}^{-1}\left(s_{M}\right) \right )^{-1} \ldots \\
	&\qquad  N \mathcal{K}^{-1}\left(s_2\right) \left ( I_n- F \mathcal{K}^{-1}\left(s_2\right) \right )^{-1}   N  \left ( I_n- \mathcal{K}^{-1}\left(s_1\right) F\right )^{-1}\Bigg).
	\end{split}
	\end{alignat*}
	We first write $U_{M+1}$ in terms of $U_M$. Thus, in the above equation using Neumann series, we get
	\begin{alignat*}{2}
	U_{M+1} &= \mathcal{K}\left ( s_1 \right )\ldots \mathcal{K}\left(s_M\right) 
	\Bigg( N \mathcal{K}^{-1}\left(s_{M}\right) \ldots N\mathcal{K}^{-1}\left(s_2\right) N   \\
	&\quad -\left ( I_n + F \mathcal{K}^{-1}\left(s_{M+1}\right) + \left(F \mathcal{K}^{-1}\left(s_{M+1}\right)\right)^2 + \cdots \right ) \\
	&\qquad N \mathcal{K}^{-1}\left(s_{M}\right) \left ( I_n- F \mathcal{K}^{-1}\left(s_{M}\right) \right )^{-1} \ldots   N \mathcal{K}^{-1}\left(s_2\right) \left ( I_n- F \mathcal{K}^{-1}\left(s_2\right) \right )^{-1}  \\
	& \qquad	 N  \left ( I_n- \mathcal{K}^{-1}\left(s_1\right) F\right )^{-1}\Bigg)\\
	&= \mathcal{K}\left ( s_1 \right )\ldots \mathcal{K}\left(s_M\right) 
	\Bigg( N \mathcal{K}^{-1}\left(s_{M}\right) \ldots N\mathcal{K}^{-1}\left(s_2\right) N   \\
	&\quad - N \mathcal{K}^{-1}\left(s_{M}\right) \left ( I_n- F \mathcal{K}^{-1}\left(s_{M}\right) \right )^{-1} \ldots   N \mathcal{K}^{-1}\left(s_2\right) \left ( I_n- F \mathcal{K}^{-1}\left(s_2\right) \right )^{-1} \\
	&\qquad N  \left ( I_n- \mathcal{K}^{-1}\left(s_1\right) F\right )^{-1}\\
	&\quad - F \mathcal{K}^{-1}\left(s_{M+1}\right) \left ( I_n - F \mathcal{K}^{-1}\left(s_{M+1}\right) \right )^{-1} \\
	&\qquad N \mathcal{K}^{-1}\left(s_{M}\right) \left ( I_n- F \mathcal{K}^{-1}\left(s_{M}\right) \right )^{-1} \ldots   N \mathcal{K}^{-1}\left(s_2\right) \left ( I_n- F \mathcal{K}^{-1}\left(s_2\right) \right )^{-1}  \\
	& \qquad	 N  \left ( I_n- \mathcal{K}^{-1}\left(s_1\right) F\right )^{-1}\Bigg).
	\end{alignat*}
	In the above equation, taking $N \mathcal{K}^{-1}\left(s_M\right)$ common from the first two terms of the bigger bracket, we have 
	\begin{alignat}{2} \label{eq:U_M+1_expression_comparisson}
	\begin{split}
	&= \mathcal{K}\left ( s_1 \right )\ldots \mathcal{K}\left(s_M\right) 
	\Bigg( N \mathcal{K}^{-1}\left(s_{M}\right) \bigg( N \mathcal{K}^{-1}\left(s_{M-1}\right)\ldots N\mathcal{K}^{-1}\left(s_2\right)N\\
	&\ - \left ( I_n- F \mathcal{K}^{-1}\left(s_{M}\right) \right )^{-1}   N \mathcal{K}^{-1}\left(s_{M-1}\right)
	\left ( I_n- F \mathcal{K}^{-1}\left(s_{M-1}\right) \right )^{-1} \ldots  \\
	&\quad N \mathcal{K}^{-1}\left(s_2\right) \left ( I_n- F \mathcal{K}^{-1}\left(s_2\right) \right )^{-1}  N  \left ( I_n- \mathcal{K}^{-1}\left(s_1\right) F\right )^{-1}\bigg)\\
	&\ - F \mathcal{K}^{-1}\left(s_{M+1}\right) \left ( I_n - F \mathcal{K}^{-1}\left(s_{M+1}\right) \right )^{-1} \\
	&\quad N \mathcal{K}^{-1}\left(s_{M}\right) \left ( I_n- F \mathcal{K}^{-1}\left(s_{M}\right) \right )^{-1} \ldots   N \mathcal{K}^{-1}\left(s_2\right) \left ( I_n- F \mathcal{K}^{-1}\left(s_2\right) \right )^{-1}  \\
	& \quad	 N  \left ( I_n- \mathcal{K}^{-1}\left(s_1\right) F\right )^{-1}\Bigg).
	\end{split}
	\end{alignat}
	Now we look at expression of $U_M$. Multiplying $\mathcal{K}^{-1} \left ( s_{M-1} \right )\ldots \mathcal{K}^{-1}\left ( s_{1} \right )$ on both the sides of \eqref{eq:U_expression_induction} from left, we get
	\begin{alignat}{2}\label{eq:U_M_expression_rewritten}
	\begin{split}
	&\mathcal{K}^{-1} \left ( s_{M-1} \right )\ldots \mathcal{K}^{-1}\left ( s_{1} \right )  U_M \\ 
	&=\Bigg( N \mathcal{K}^{-1}\left(s_{M-1}\right) \ldots N\mathcal{K}^{-1}\left(s_2\right) N   \\
	&\ - \left ( I_n- F \mathcal{K}^{-1}\left(s_M\right) \right )^{-1}   N \mathcal{K}^{-1}\left(s_{M-1}\right) \left ( I_n- F \mathcal{K}^{-1}\left(s_{M-1}\right) \right )^{-1} \ldots \\
	&\qquad  N \mathcal{K}^{-1}\left(s_2\right) \left ( I_n- F \mathcal{K}^{-1}\left(s_2\right) \right )^{-1}   N  \left ( I_n- \mathcal{K}^{-1}\left(s_1\right) F\right )^{-1}\Bigg).
	\end{split}
	\end{alignat}
	Substituting \eqref{eq:U_M_expression_rewritten} in \eqref{eq:U_M+1_expression_comparisson}, we get
	\begin{alignat*}{2} 
	\begin{split}
	U_{M+1} &= \mathcal{K}\left ( s_1 \right )\ldots \mathcal{K}\left(s_M\right) 
	\Bigg( N \mathcal{K}^{-1}\left(s_{M}\right)  \bigg(\mathcal{K}^{-1} \left ( s_{M-1} \right )\ldots \mathcal{K}^{-1}\left ( s_{1} \right )  U_M\bigg)\\
	&\quad - F \mathcal{K}^{-1}\left(s_{M+1}\right) \left ( I_n - F \mathcal{K}^{-1}\left(s_{M+1}\right) \right )^{-1} \\
	&\qquad N \mathcal{K}^{-1}\left(s_{M}\right) \left ( I_n- F \mathcal{K}^{-1}\left(s_{M}\right) \right )^{-1} \ldots   N \mathcal{K}^{-1}\left(s_2\right) \left ( I_n- F \mathcal{K}^{-1}\left(s_2\right) \right )^{-1}  \\
	& \qquad	 N  \left ( I_n- \mathcal{K}^{-1}\left(s_1\right) F\right )^{-1}\Bigg).
	\end{split}
	\end{alignat*}
	Taking $H_{\infty}-$norm on both sides, we get
	\begin{alignat*}{2} 
	&\left \| U_{M+1}\right \|_{H_{\infty}}  = \\
	& \quad \left \| \mathcal{K}\left ( s_1 \right )\ldots \mathcal{K}\left(s_M\right) 
	\Bigg( N \mathcal{K}^{-1}\left(s_{M}\right) \mathcal{K}^{-1} \left ( s_{M-1} \right )\ldots \mathcal{K}^{-1}\left ( s_{1} \right )  U_M  \right.\\
	&\quad - F \mathcal{K}^{-1}\left(s_{M+1}\right) \left ( I_n - F \mathcal{K}^{-1}\left(s_{M+1}\right) \right )^{-1} \\
	&\qquad N \mathcal{K}^{-1}\left(s_{M}\right) \left ( I_n- F \mathcal{K}^{-1}\left(s_{M}\right) \right )^{-1} \ldots   N \mathcal{K}^{-1}\left(s_2\right)  \left ( I_n- F \mathcal{K}^{-1}\left(s_2\right) \right )^{-1}\\
	& \left. \qquad	  N  \left ( I_n- \mathcal{K}^{-1}\left(s_1\right) F\right )^{-1}\Bigg)\right \|_{H_{\infty}}. 
	\end{alignat*}
	As earlier, using the norm inequality properties in the above equation, we get
	\begin{alignat*}{2}
	&\left \| U_{M+1} \right \|_{H_{\infty}} 
	\leq \\
	&\qquad \underset{\omega_1, \ \ldots, \ \omega_{M+1} \in \mathbb{R}}{max} \Bigg[ \left \| \mathcal{K}\left ( i\omega_1 \right ) \right \|_2 \ldots \left \| \mathcal{K}\left(i\omega_M\right) \right \|_2 
	\bigg( \left \| N \right \|_2 \left \|  \mathcal{K}^{-1}\left(i\omega_M\right)  \right \|_2  \ldots   \\
	&\qquad  \left \| \mathcal{K}^{-1}\left ( i\omega_1 \right )  \right \|_2 \left \| U \left(i\omega_1, \ \ldots, \ i\omega_M\right)\right \|_2 + \left \| F \right \|_2 \left \| \mathcal{K}^{-1}\left(i\omega_{M+1}\right) \right \|_2   \\
	&\qquad \left \| \left ( I_n - F \mathcal{K}^{-1}\left(i\omega_{M+1}\right) \right )^{-1}  \right \|_2 \left \| N \right \|_2 \left \| \mathcal{K}^{-1}\left(i\omega_M\right)  \right \|_2  \left \| \left ( I_n- F \mathcal{K}^{-1}\left(i\omega_M\right) \right )^{-1} \right \|_2 \\
	&\qquad \ldots   \left \| N \right \|_2 \left \| \mathcal{K}^{-1}\left(i\omega_2\right) \right \|_2  \left \| \left ( I_n- F \mathcal{K}^{-1}\left(i\omega_2\right) \right )^{-1}   \right \|_2  \\
	& \qquad  \left \| N \right \|_2  \left \|  \left ( I_n- \mathcal{K}^{-1}\left(i\omega_1\right) F\right )^{-1}\right \|_2 \bigg) \Bigg].
	\end{alignat*}
	Similar to \eqref{eq:max_to_H_infty_inequality_golub_s2} and \eqref{eq:golub_inequality_rational_form_individual_2_infty_norm_for_s2}, here also, using Lemma 2.3.3 from \cite{g_golub_book} we get 
	\begin{align*}
	\left \| U_{M+1} \right \|_{H_{\infty}} 
	\leq &  \left \| \mathcal{K}\left(s_1\right)\right \|_{H_{\infty}} 
	\ldots \left \| \mathcal{K}\left(s_{M}\right)\right \|_{H_{\infty}} \left \| N \right \|_{2} \left \| \mathcal{K}^{-1}\left(s_{M}\right)\right \|_{H_{\infty}}  \ldots  \\ & \left \| \mathcal{K}^{-1} \left(s_2\right)\right \|_{H_{\infty}} \Bigg[  \left \|\mathcal{K}^{-1} \left(s_1\right)\right \|_{H_{\infty}}  \left \|U_M\right \|_{H_{\infty}} +   \\
	&\dfrac{ \left \|N\right \|_{2}^{M-1}  \left \| \mathcal{K}^{-1}\left(s_{M+1}\right)\right \|_{H_{\infty}}  }{\left(1-\left \|  F \right \|_{2} \left \|  \mathcal{K}^{-1}\left(s_{M+1}\right) \right \|_{H_{\infty}}\right) \ldots \left(1-\left \|  F \right \|_{2} \left \|  \mathcal{K}^{-1}\left(s_2\right) \right \|_{H_{\infty}}\right)} \cdot \\
	&\dfrac{\left \|   F  \right \|_{2}}{1-\left \| \mathcal{K}^{-1} \left(s_1\right)    \right \|_{H_{\infty}} \left \|   F  \right \|_{2}} \Bigg].
	\end{align*}
	From induction hypothesis we know $\left \| U_{M} \right \|_{H_{\infty}}  \propto \mathcal{O} \left ( \left \| F \right \|_2 \right )$. Using this we get
	\begin{align*}
	\left \| U_{M+1} \right \|_{H_{\infty}}  \propto \mathcal{O} \left ( \left \| F \right \|_2 \right ). 
	\end{align*}
\end{proof}
\begin{theorem}\label{Theorem:second_condition_subsystem}
If 
\begin{align*}
H_{M}\left ( s_{1},\ \ldots,\ s_{M} \right )  = & \ c\left ( s_{M}I_n-A \right )^{-1}N \left ( s_{M-1}I_n-A \right )^{-1}      \ldots N\left ( s_{1}I_n-A \right )^{-1}b,\\
\widetilde{H}_{M}\left ( s_{1},\ \ldots,\ s_{M} \right )  = & \ c\left ( s_{M}I_n-\left ( A+F \right ) \right )^{-1} \\
& \ N \left ( s_{M-1}I_n- \left ( A+F \right ) \right )^{-1} \ldots   N\left ( s_{1}I_n- \left ( A+F \right ) \right )^{-1}b,
\end{align*}
	$\left\|F\right\|_2 < 1$, and  $\left \|\mathcal{K}^{-1}\left(s\right)  \right \|_{H_{\infty}} <1$, where $\mathcal{K}\left(s\right)= \left(sI_n - A\right)$, then
	\begin{align*}
	\left \|H_{M}\left ( s_{1},\ \ldots,\ s_{M} \right ) - \widetilde{H}_{M}\left ( s_{1},\ \ldots,\ s_{M} \right )  \right \|_{H_{2}}^2 = \mathcal{O}\left ( \left \| F \right \|_2^2\right ) \qquad \textnormal{or}
	\end{align*}
	TBIRKA satisfies the second condition of backward stability with respect to inexact linear solves.   
\end{theorem}
\section[Conclusions \& Future Work]{Conclusions \& Future Work} \label{sec:Conclusions Future Work}
In this paper, we apply iterative linear solvers during model order reduction (MOR) of bilinear dynamical systems. Since such solvers are inexact, the stability of the underlying MOR algorithm, with respect to these approximation errors, is important. Here, we extend the earlier stability analysis done for BIRKA in \cite{BIRKAstabilitylaapaper}, which is a standard algorithm for obtaining a reduced bilinear dynamical system, to TBIRKA, a cheaper variant of BIRKA. 
\par 
Proving that an algorithm is stable, typically requires satisfying two conditions. In TBIRKA, fulfilling the first condition for stability leads to constraints on the iterative linear solver, which are similar to those obtained during BIRKA's stability analysis. The second condition for a stable TBIRKA is satisfied using an approach different than the one used in BIRKA, and is more intuitive.
\par 
Our first future direction is to extend our analysis from SISO (Single Input Single Output) to MIMO (Multiple Input Multiple Output) systems. The stability analysis as done for BIRKA earlier and TBIRKA here, all give us sufficiency conditions for a stable underlying MOR algorithm. Hence, second, we plan to derive the necessary conditions for the same. In recent years, there have been a lot of efforts in performing data-driven MOR algorithm (specially using Leowner framework \cite{bilinear_model_reduction_Antoulas_Loewner_framework}). Similar linear systems and challenges associated with the scaling arise here as well. Our third future direction is to apply this stability analysis to such classes of algorithms as well. 


\vspace{2mm} \noindent \footnotesize
\begin{minipage}[b]{10cm}
Rajendra Choudhary, \\
Computational Science \& Engineering Lab, \\ 
Indian Institute of Technology Indore, \\ 
Khandwa Road, Simrol, Indore-453552, India. \\
Email: rajendracse46@gmail.com, phd1301201004@iiti.ac.in
\end{minipage}

\vspace{2mm} \noindent \footnotesize
\begin{minipage}[b]{10cm}
Kapil Ahuja, \\
Computational Science \& Engineering Lab, \\  
Indian Institute of Technology Indore, \\ 
Khandwa Road, Simrol, Indore-453552, India. \\
Email: kapsahuja22@gmail.com, kahuja@iiti.ac.in
\end{minipage}

\end{document}